\documentclass[11pt]{article}
\usepackage[utf8]{inputenc}
\usepackage[T1]{fontenc}

\usepackage{amsmath}
\usepackage{amsfonts}
\usepackage{ amssymb }
\usepackage{ mathrsfs }
\usepackage{amsthm}
\usepackage{bm} 

\usepackage{dsfont}
\usepackage{xcolor}
\usepackage{fancyvrb}

\usepackage{graphicx}
\graphicspath{{Images/}}
\usepackage{hyperref}
\usepackage{thmtools}
\usepackage{thm-restate}

\usepackage[lmargin=1in,rmargin=1in,bmargin=1.5in]{geometry}

\usepackage{todonotes}

\declaretheorem[name=Theorem,numberwithin=section]{thm}

\newtheorem{theorem}[thm]{Theorem}
\newtheorem{proposition}[thm]{Proposition}
\newtheorem{lemma}[thm]{Lemma}

\theoremstyle{definition}
\newtheorem{definition}[thm]{Definition}
\newtheorem{example}[thm]{Example}

\numberwithin{equation}{section}

\newcommand{\qbin}[2]{\genfrac[]{0pt}{}{#1}{#2}}

\newcommand{\bZ}{\mathbb{Z}}
\newcommand{\dP}{\mathds{P}}

\newcommand{\ms}{\mathrm{MSet}}
\newcommand{\cst}{\mathrm{cst}}

\usepackage[backend=biber]{biblatex}
\begin{document}

\title{Remixed Eulerian numbers: beyond the connected case}

\author{Solal Gaudin}


\maketitle

\abstract{In his study of generalised permutahedra, Postnikov considered the mixed volumes of hypersimplices, giving rise to the family of mixed Eulerian numbers. It comprises usual Eulerian numbers, binomial coefficients, Catalan numbers, and the large family of hit numbers. Nadeau and Tewari further gave a polynomial refinement of these, the remixed Eulerian numbers, which recover the natural $q$-analogs of these special families. Using a probabilistic model, we give several formulas for remixed Eulerian numbers for certain subfamilies, extending the known formulas for $q$-hit numbers due to Garsia and Remmel.

\section{Introduction}
Eulerian numbers have long been a core part of both enumerative  and algebraic combinatorics. Petersen dedicated a whole book to them \cite{Pet15}. While studying the volumes of the permutahedra, Postnikov defined the family of \textit{mixed Eulerian numbers} $A_c$ \cite{Pos09}. These numbers generalise not only Eulerian numbers, but also binomial coefficients, Catalan numbers, and others; they also satisfy several nice enumerative properties. Postnikov gave a combinatorial interpretation as counting certain weighted trees, while Liu later showed that they also counted certain special permutations \cite{Liu16} (see also~\cite[Sec. 7]{NT22}). 
\medskip

This family has since been expanded in multiple directions. Postnikov already introduced mixed Eulerian numbers in all Lie types; the original ones correspond to type $A$. Horiguchi gave combinatorial interpretation for those in classical types \cite{Hor23}. In 2021, Nadeau and Tewari gave a polynomial $q$-analog of the $A_c$: the \emph{remixed Eulerian numbers} $A_c(q)$ \cite{NT21}. They then showed that most properties of the usual mixed Eulerian numbers extended nicely \cite[Theorem 1.3]{NT22}. Recently, Katz and Kutler defined general \textit{matroidal} mixed Eulerian numbers \cite{KatKut24}: when $q$ is a prime power and the matroid is the projective field over $\mathbb{F}_q$, these coincide with remixed Eulerian numbers. 
\medskip

We are interested here in the remixed Eulerian numbers $A_c(q)$. Here, $c$ is a tuple $(c_i)_{i \in [1;n]}$ of integers $c_i\geq 0$ such that $\sum_i{c_i}= n$. Nadeau and Tewari proved that they are a family of symmetrical, unimodal polynomials in a parameter $q$ \cite{NT22} with nonnegative coefficients. Furthermore, these have been shown in \cite{NT22,Mit20} to be a generalisation of $q$-binomial numbers, of Garsia-Remmel's $q$-hit numbers defined in \cite{GarRem86}, and of $q$-Eulerian numbers. The remixed Eulerian numbers were historically defined algebraically \cite{NT21,BerSpiTse23}. Here we will use the simple probabilistic definition later introduced by Nadeau and Tewari \cite{NT22}.
\bigskip

\noindent{\bf{Probabilistic definition of $A_c(q)$}.} Here we let $q$ be a real non-negative number. Imagine $n$ balls on sites $1,\ldots,n$, and no ball on the other sites of $\bZ$. We can encode this by a configuration $c$ as above, where $c_i$ is the number of balls at size $i$. Pick any ball at a site $i$ with $c_i>1$, and move it left with probability $q/(1+q)$ and right with probability $1/(1+q)$. We then obtain a new configuration $c'$. Repeat this operation on any site with at least two balls, until this is no more such site. Let $\dP_c(q)$ be the probability that all particles are still within $[1;n]$ at this time, that is each site contains one particle. Then $\dP_c(q)$ is well-defined, and we set $A_c(q)=[n]!\dP_c(q)$.

We call \emph{support} the set of $i$ such that $c_i > 0$.
\begin{figure}[!ht]
\label{fg:ex_config}
\centering
    \includegraphics[width=0.35\textwidth]{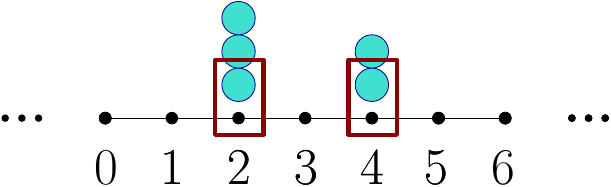}
\caption{An example of configuration having $n = 5$ balls, with its support in full lines.\label{fg:support}}
\end{figure}
\medskip

\noindent{\bf{Results.}}
Previous work on mixed Eulerian numbers has mostly focused on the case where the support is connected \cite{Mit20,KatKut24,NT22}. This case turns out to be closely linked with the study of Garsia and Remmel's $q$-hit numbers \cite{NT22}, see \cite{CMP21} for recent work over these $q$-hit numbers and their relation with chromatic symmetric functions.

In this paper, we give elementary proofs of the known results for mixed Eulerian numbers. The directness of these proofs allows us to extend them beyond connected configurations, in Sections \ref{sec:almost_luka} and \ref{sec:wLuka}. Intuitively, whenever the set of fallen balls stays connected throughout the dynamic, the induction relation \eqref{eq:fin_ind} is simplified. This observation drives the choices of which extended families of configurations we study in Sections \ref{sec:almost_luka} and \ref{sec:wLuka}.
\medskip

We first recall basic properties of this particle model. Thanks to a result of Diaconis and Fulton \cite{DiaFul91}, we know that the remixed Eulerian numbers $A_c(q)$ are well-defined. Nadeau and Tewari further gave them two induction relations detailed in Section \ref{sec:state_art}, and proved explicit formulas for some families of configurations \cite{NT22}.

These two families are the \emph{{\L}ukasiewicz} configurations and \emph{connected} configurations. We dedicate Sections \ref{sec:Luka} and \ref{sec:connected} to them, recalling their strong links with $q$-Eulerian numbers and $q$-hit numbers. We build the rest of this paper on these two families, their specialisations and generalisations.

We link Carlitz and Scoville's generalised Eulerian numbers \cite{CarSco74} to a specialisation of connected configurations in Section \ref{sec:ab-eul}. As such, we are able to give a $q$-analog of these generalised Eulerian numbers. In Section \ref{sec:almost_luka}, we present a variation of {\L}ukasiewicz configurations. We are able to give their remixed Eulerian number a formula, with Theorem \ref{th:alm_luka}.

In Section \ref{sec:wLuka}, we unite connected configurations and {\L}ukasiewicz configurations, with \emph{weakly {\L}ukasiewicz} configurations. For convenience's sake, we define $\ms(c)$ the multiset corresponding to a configuration $c$, having $c_i$ times each element $i \in [1;n]$. One of our main results then is the closed formula for their remixed Eulerian number:
\begin{restatable*}{thm}{thwluka}
\label{th:w_luka}
Let $k$ be the maximal value such that $(0^k,\gamma,0^{n-m-k})$ is a weakly {\L}ukasiewicz configuration. Then:
\begin{equation}
    \sum_{i=0}^{n-m}{t^i A_{(0^i,\gamma,0^{n-m-i})}(q)} \equiv {(t;q)_{n+1}} \sum_{j \geq 0}{t^j \prod_{a \in \ms(\gamma)}{[j+a]}} \qquad \mod t^{k+1}.
\end{equation}
By extracting the coefficient of $t^k$, we get:
\begin{equation}
    A_{(0^k,\gamma,0^{n-m-k})}(q) = \sum_{j = 0}^{k}{\left( (-1)^{j+k}q^{\binom{k-j}{2}}\qbin{n+1}{k-j} \prod_{a \in \ms(\gamma)}{[j+a]} \right) }.
\end{equation}
\end{restatable*}

Finally, in Section \ref{sec:two_connec}, we specialise weakly {\L}ukasiewicz configurations into \emph{one-hole} configurations: configurations with two connected parts, barely apart from one another. The theorem \ref{th:one_hole} gives their remixed Eulerian number, using new techniques for its proof.

\section{A particle model}
\label{sec:state_art}
We give here a full probabilistic definition of the remixed Eulerian numbers, given by Petrov \cite{Pet18}, and later refined by Nadeau and Tewari \cite[Sec. 2.3]{NT22}. We write each remixed Eulerian number as $A_c(q)$. They thus depend on a non-negative real $q$ and on a \emph{configuration} $c$: a tuple $(c_i)_{i \in [1;n]}$ of integers $c_i\geq 0$,  such that $\sum_i{c_i}= n$. We define $\ms(c)$  the corresponding multiset, having $c_i$ times each element $i \in [1;n]$. We call \emph{balls} these $n$ elements, and \emph{support} the set of $i$ such that $c_i > 0$, as in Figure \ref{fg:ex_config}.

The model itself consists of the finite evolution of these $n$ balls, whose set of possible \emph{sites} is the integer line $\bZ$. The balls start over that line, and are dropped onto it one at a time, in a chosen order $u_1,...,u_n$. As an example, for $c = (0,3,0,2,0)$, the multiset of balls is $\{ \! \{ 2,2,2,4,4 \} \! \}$, and we might choose to drop them in the order $2,4,4,2,2$. We say that $u_1,...,u_n$ has \emph{content} $c$.

When a ball lands on top of another ball, it moves one site left (resp. right) with probability $q/(1+q)$ (resp. $1/(1+q)$). We repeat this step until the ball arrives at an empty site: a \emph{hole}.
\begin{figure}[!ht]
\centering
    \includegraphics[width=0.35\textwidth]{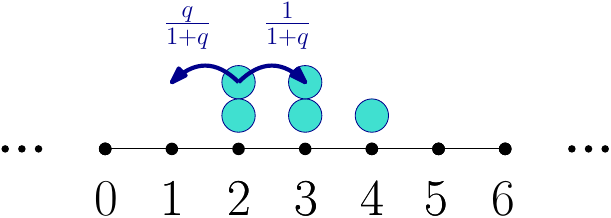}
\caption{Moving a ball one site.\label{fg:ppas}}
\end{figure}

This defines a finite dynamic, of which we wish to know more. We first notice that at the end of the dynamic, no two balls have stopped at the same place. We thus ask: what is the probability $\dP(c \to Supp)_q$ that a given set $Supp$ is the final support, when the balls started in the configuration $c$?

\subsection{First results}
\label{sec:known_basics}
We wish to know the probability that a given set $Supp$ is the final support, when the balls started in the given configuration $c$. A classical result from Diaconis and Fulton \cite{DiaFul91} ensures that these probabilities are indeed  independent from the order in which the balls drop, and thus are well-defined. In other words, the following operations commute with each other and between themselves: "dropping a ball", and "moving it one space -- randomly left or right".
\begin{figure}[!ht]
\centering
    \includegraphics[width=0.35\textwidth]{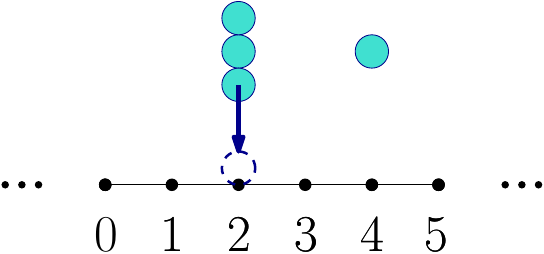}
\caption{A ball, dropping onto $\bZ$.\label{fg:drop}}
\end{figure}

Nadeau and Tewari first studied these probabilities in \cite{NT22}, and proved that we can easily obtain them for any set $Supp$ once we have them for all intervals $[1;i]$. We can thus suppose w.l.o.g. that $Supp = [1;n]$, and we write $\dP(c \to [1;n])_q$ the probability that the dynamic ends with such $Supp = [1;n]$, from the given configuration $c$.

Furthermore, these authors proved \cite[Prop. 5.4]{NT21} that multiplying by a quantity -- only depending on the \emph{number} of balls -- is enough to turn these probabilities from rational fractions to \emph{polynomials} in $q$. This quantity is the $q$-factorial $[n]!$. We also define the $q$-integers and $q$-binomials, for further use:
$$[n]! = \prod_{i = 1}^n{[i]} \qquad ; \qquad [i] := \sum_{k = 0}^{i-1}{q^k} \qquad ; \qquad \qbin{n}{k} = \frac{[n]!}{[k]![n-k]!}.$$

We can now define:
\begin{definition}
Let $c$ be a configuration. Then, the remixed Eulerian number is: $$A_c(q) = [n]! \dP(c \to [1;n])_q.$$
\end{definition}
We will simply write $A_c$ when the choice of $q$ is clear. Conversely, we will always write $A_c(1)$ to talk about regular mixed Eulerian numbers.
\smallskip

We now call the configuration $c+\{j\}$ as the configuration $c$ with one more ball at site $j$ -- that is, $\ms(c+\{j\}) = \ms(c) \cup \{j \}$. Similarly, if $c_j > 0$, we call $c-\{ j \}$ the configuration $c$ with one more fewer ball at site $j$, i.e. $\ms(c-\{j\}) = \ms(c) \backslash \{j \}$.
\smallskip

One can move the balls multiple sites at the time, instead of moving only one space with probability $q/(1+q)$ and $1/(1+q)$. The resulting probabilities are given by the following proposition:
\begin{proposition}[\cite{NT22}, Eq (2.4)]
Take $j \in [1;n]$ such that the nearest hole -- in $c$ -- on its left (resp. right) is at distance $a \geq 0$ (resp. $b \geq 0$). We then have:
\begin{equation}
\label{eq:big_step}
A_{c+\{j\}} = q^a \frac{[b]}{[a+b]}A_{c+\{j-a\}} + \frac{[a]}{[a+b]}A_{c+\{j+b\}}. 
\end{equation}
\end{proposition}
This can be seen as a typical one-dimensional drunk walk, as studied in \cite{Fel68}.
\begin{figure}[!ht]
\centering
    \includegraphics[width=0.35\textwidth]{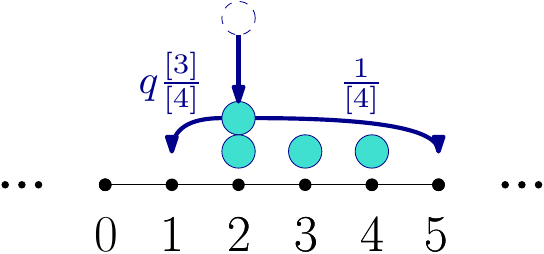}
\caption{A ball moving multiple sites at once.\label{fg:gpas}}
\end{figure}

After defining the \emph{reverse configuration} $\overline{c} := (c_n,...,c_1)$, we have:

\begin{proposition}[\cite{NT22}, Lemma 5.1]
\label{pr:reverse}
$A_c(q) = q^{\binom{n}{2}}A_{\overline{c}}(\frac{1}{q}).$
\end{proposition}
\begin{proof}
    The symmetry of centre $\frac{n+1}{2}$ exchange each site $i \leq \frac{n+1}{2}$ with the site $n+1-i$. The configuration $c$ is now mirrored. Furthermore, swapping the probabilities $q/(1+q)$ and $1/(1+q)$ is equivalent to replacing $q$ by $1+q$. Therefore $\dP(c)_q = \dP(\overline{c})_{1/q}$.
    
    By the definition of remixed Eulerian numbers, we know that $A_{\overline{c}}(1/q) = [n]_{1/q}! \dP(\overline{c})_{1/q}$ and $A_c(q) = [n]_q! \dP(c)_q$, so $A_{\overline{c}}(\frac{1}{q}) [n]_q! = A_c(q) [n]_{1/q}!.$
    
    Finally, we recall the classical identity $[n]_q! = q^{\binom{n}{2}} [n]_{1/q}!.$
\end{proof}

\subsection{Another induction relation}
\label{sec:fin_ind}
For any configuration $c$ and any index $j$, we want to study how the configuration $c$ can get to $[1;n]$, with the added condition that the last ball must fall onto site $j$. Let $u_1,u_2,...,u_n$ be an order of content $c$. In order to highlight the relevancy of the starting point $u_n$ of the last ball dropped, we write $(c-\{ u_n \}) + \{ u_n \}$ instead of $c$.

In order for the last ball to fall onto site $j$, exactly $j-1$ of the balls in $c-\{ u_n \}$ must arrive in $[1;j-1]$, with no ball ever going onto $j$. This is only possible if this same amount of these balls \emph{started} in $[1;j-1].$ We define the \emph{height} of a configuration $c$ at index $k$ as: 
\begin{equation}
H_{c,k} = \sum_{i = 1}^{k}{(c_i-1)}.
\end{equation}

We finally have the induction at the final step, or \emph{final induction} \cite[Eq (3.1)]{NT22}:
\begin{align}
\label{eq:fin_ind}
    A_c &= \sum_{k \text{ s.t. }  H_{c-\{ u_n \},k} = 0} wt(k,u_n) A_{(c_1,...,c_{k-1})} A_{(c_{k+1},...,c_n)},
\end{align}
\begin{align*}
    \text{with } wt(k,u_n) = \left\lbrace
    \begin{array}{ll} \qbin{n}{k}[u_n] &\text{if } k \geq u_n,\\
        q^{u_n-k} \qbin{n}{k-1}[n+1-u_n] &\text{if } k < u_n.
    \end{array}\right.
\end{align*}

The previous induction is especially easy to handle when the sum has at most two terms. This fact was used for two specific classes \cite{NT22}, to which we dedicate Sections \ref{sec:Luka} and \ref{sec:connected}.

\section{{\L}ukasiewicz configurations and connected configurations}
\label{sec:known_fam}

We define in this section two families of configurations. Nadeau and Tewari gave the remixed Eulerian numbers for both of these families in \cite[Sec. 4]{NT22}. These families, and these formulas, will guide us towards the definitions of families of Sections \ref{sec:almost_luka} and \ref{sec:wLuka}, and the proofs of their respective formulas.

\subsection{{\L}ukasiewicz configurations}
\label{sec:Luka}
\begin{definition}
    The \emph{left-to-right order} of a configuration $c$ is the only non-decreasing sequence $u_1,...,u_n$ of content $c$.
\end{definition}
\begin{definition}
A configuration $c$ is \emph{{\L}ukasiewicz} if all its heights $H_{c,i}, i \in [1;n]$ are non-negative.
\end{definition}
\begin{proposition}[\cite{NT22}, Prop 4.2]
\label{pr:acq_luka}
Let $c$ be a {\L}ukasiewicz configuration. Then:
$$A_c = \prod_{a \in \ms(c)}{[a]}.$$
In other words, in order for a {\L}ukasiewicz configuration to end up into $[1;n]$ when balls are dropped left-to-right, all of the balls need to fall on the right.
\end{proposition}
We present here a short proof, that we will generalise in Section \ref{sec:almost_luka} to other configurations.
\begin{proof}
    We take the left-to-right order $u_1, ..., u_n$ of the configuration $c$. Using the final induction \eqref{eq:fin_ind}, we have $A_{u_1,...,u_n} = \qbin{n}{n} [u_n] A_{u_1,...,u_{n-1}}$. By induction, we thus have the equality $A_{u_1,...,u_n} = \prod_{i \in [1;n]}{[u_i]},$ QED.
\end{proof}

\begin{example}
The configuration $(3,0,0,2,0)$ is {\L}ukasiewicz, as its heights are $(2,1,0,1,0).$ Therefore, $A_{(3,0,0,2,0)} = [1]^3 [4]^2 = 1+2q+3q^2+4q^3+3q^4+2q^5+q^6.$
\end{example}
\begin{figure}[!ht]
\centering
    \includegraphics[width=0.35\textwidth]{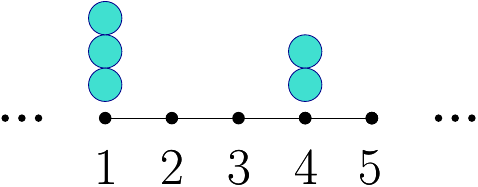}
\caption{A {\L}ukasiewicz configuration. \label{fg:luka}}
\end{figure}

The original statement of this result by Nadeau and Tewari \cite[Eq (5.1)]{NT22} is slightly more general, but we will generalise it in other directions, in Sections \ref{sec:almost_luka} and \ref{sec:wLuka}.

\subsection{Connected configurations}
\label{sec:connected}
A generic configuration $c$, of size $n > 0$, might have many holes on its leftmost and rightmost parts. We can thus write it, in exactly one way, as $(0^k,\gamma,0^{n-m-k})$ with $m > 0$, $\gamma = (\gamma_1,...,\gamma_m)$ and $\gamma_1, \gamma_m > 0$.
\begin{definition}
\label{df:core_part}
We call such a $\gamma$ the \emph{core} of the configuration $c$.
\end{definition}

As $\gamma$ may have fewer sites than balls, it is not a configuration per se, but rather a \emph{sub-configuration}.

\begin{figure}[!ht]
\centering
    \includegraphics[width=0.35\textwidth]{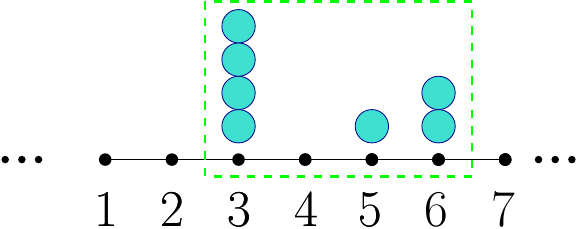}
\caption{In dashed lines: the core of a configuration.\label{fg:core}}
\end{figure}

\begin{definition}
When this core part $\gamma$ has at least one ball per site, i.e. no hole, then we say that $c$ is \emph{connected}.
\end{definition}
As an example, the configuration $c = (0,1,2,2,0)$ has core $(1,2,2)$ and is thus connected. Conversely, the configuration $c = (0,3,0,2,0)$ has core $(3,0,2)$, so it is not connected.

\begin{figure}[!ht]
\centering
    \includegraphics[width=0.35\textwidth]{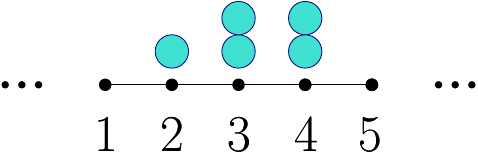}
\caption{A connected configuration.\label{fg:connected}}
\end{figure}
\begin{definition}
The \emph{q-Pochhammer} $(t;q)_{n}$ is a $q$-analog of $(1-t)^n$: $(t;q)_{n} = \prod_{i=0}^{n-1}{(1-tq^i)}.$ \end{definition}
Moreover, we have the $q$-binomial theorem \cite[Thm 3.3]{And84}: 
\begin{equation}
\label{eq:q-poch_coeff}
(t;q)_n = \sum_{j=0}^n{(-1)^j t^j q^{\binom{j}{2}} \qbin{n}{j}}.
\end{equation}

We can then give closed formulas of $A_c(q)$ for any connected configuration $c$, written as $(0^k,\gamma,0^{n-m-k})$.

\begin{theorem}(\cite{NT22}, Eq (4.2))
\label{th:acq_cnx}
Let $\gamma$ be a sub-configuration with no hole, with $n$ total balls on $m$ sites. Then:
\begin{equation}
    \label{eq:acq-cnx-qpoch}
    \frac{\sum_{i=0}^{n-m}{t^i A_{(0^i,\gamma,0^{n-m-i})}}}{(t;q)_{n+1}} = \sum_{j \geq 0}{t^j \prod_{a \in \ms(\gamma)}{[j+a]}}.
\end{equation}

Extracting a single coefficient of the numerator thanks to Equality \eqref{eq:q-poch_coeff}, we get that for all $i \in [0;n-m]$:
\begin{equation}
    \label{eq:acq-cnx-1coeff}
    A_{(0^i,\gamma,0^{n-m-i})} = \sum_{j = 0}^{i}{(-1)^{i+j}q^{\binom{i-j}{2}}\qbin{n+1}{i-j}} \prod_{a \in \ms(\gamma)}{[j+a]}.
\end{equation}
\end{theorem}

This theorem comes from Nadeau and Tewari's paper \emph{Remixed Eulerian numbers} \cite[Eq. (4.2)]{NT22}. However, our proof is more elementary than theirs. As such, we will be able to extend it to other configurations, see Section \ref{sec:wLuka}.

\begin{proof}
    By induction over $n$. We write $c^{(i)}$ the configuration $(0^i,\gamma,0^{n-m-i-1})$, $b = m$, and finally $d^{(i)} = c^{(i)}-\{ i+b \}$. We then have, thanks to the final induction \eqref{eq:fin_ind}:
    \begin{align*}
        \sum_{i=0}^{n-m}{t^i A_{c^{(i)}}} &= \sum_{i=0}^{n-m} {A_{d^{(i)} + \{i+b\}}} \\
        &= \sum_{i=0}^{n-m}{t^i [i+b] A_{d^{(i)}}} + \sum_{i=0}^{n-m}{t^i q^{i+b-1} [n+1-i-b] A_{d^{(i-1)}}} \\
        &= \frac{1}{q-1} \left( \sum_{i = 0}^{n-m-1}{\hspace{-5pt} (tq)^i q^b A_{d^{(i)}}} - \hspace{-10pt} \sum_{i = 0}^{n-m-1}{t^i A_{d^{(i)}}} + \sum_{i = 1}^{n-m}{t^i q^n A_{d^{(i-1)}}} - \hspace{-5pt} \sum_{i = 1}^{n-m}{(tq)^i q^{b-1} A_{d^{(i-1)}}} \right) \\
        &= \frac{q^b(1-t)}{q-1}\sum_{i=0}^{n-m-1}{(tq)^i A_{d^{(i)}}} + \frac{tq^n-1}{q-1}\sum_{i=0}^{n-m-1}{t^i A_{d^{(i)}}}.
    \end{align*}
    Using the induction hypothesis, we get:
    \begin{align*}
        \sum_{i=0}^{n-m}{t^i A_{c^{(i)}}} &= \frac{q^b}{q-1}\sum_{j \geq 0}{(tq)^j (1-t)(tq;q)_n \hspace{-10pt} \prod_{a \in \ms(\gamma - \{ b \})}{\hspace{-20pt} [j+a]}} + \frac{tq^{n}-1}{q-1}\sum_{j \geq 0}{t^j (t;q)_n \hspace{-10pt} \prod_{a, a \in \ms(\gamma - \{ b \})}{\hspace{-20pt} [j+a]}} \\
        &= \sum_{j \geq 0}{t^j (t;q)_{n+1} \left( \frac{q^{b+j}}{q-1} - \frac{1}{q-1} \right) } \prod_{a \in \ms(\gamma - \{ b \})}{[j+a]} \\
        &= (t;q)_{n+1} \sum_{j \geq 0}{t^j \prod_{a \in \ms(\gamma)}{[j+a]}}.
    \end{align*}
\end{proof}

In our example $c = (0,1,2,2,0)$ from Figure \ref{fg:connected}, we thus have:
\begin{align*}
A_{(0,1,2,2,0)} &= (-1)^1 q^0 \qbin{6}{1} [1] [2]^2 [3]^2 + (-1)^2 q^0 \qbin{6}{0} [2] [3]^2 [4]^2 \\
&= [2] [3]^2 [4]^2 - [2]^2 [3]^2 [6] \\
&= q^2+5q^3+12q^4+18q^5+18q^6+12q^7+5q^8+q^9.
\end{align*}

\noindent{\bf{Garsia-Remmmel's $q$-hit numbers.}}
In \cite[Sec 4.2]{NT22}, Nadeau and Tewari proved that the family of remixed Eulerian numbers of connected configurations is exactly the family of $q$-hit numbers.

The \emph{$q$-hit number} $H_i(\lambda,q)$ was previously defined by Garsia and Remmel \cite{GarRem86}, and has been given combinatorial interpretations by Dworkin \cite{Dwo98}, Haglund and Remmel \cite{HagRem01}, then Colmenajero, Morales and Panova \cite{CMP21}. It is a polynomial in $q$, depending on parameters $q > 0, i \in [1;n]$ and $\lambda$ a partition. Furthermore, one should compare to Eq \eqref{eq:acq-cnx-qpoch} for connected configurations to the following result, due to Garsia and Remmel:

\begin{proposition}[\cite{GarRem86}]
\label{pr:qhits_gen_ser}
Let $\lambda$ be a partition of size $n$. We then have:
\begin{equation}
    \label{eq:qhit_form}
    \frac{\sum_{i=0}^{n}{t^i H_i(\lambda,q)}}{(t;q)_{n+1}} = \sum_{j \geq 0}{t^j \prod_{i = 1}^{n}[j+i-\lambda_{n+1-i}]}.
\end{equation}
\end{proposition}

From this comparison, Nadeau and Tewari proved in \cite[Section 4.2]{NT22} that, for any $q$-hit number $H_i(\lambda,q)$ of size $n$, there is a connected configuration $c$ also of size $n$, such that we have $A_c = H_i(\lambda,q)$. Conversely, from any connected configuration $c$, one can build a partition $\lambda$ and choose an integer $i$ such that $H_i(\lambda,q) = A_c.$

The family of remixed Eulerian number thus extends the $q$-hit numbers, seeing as a generic configuration need not be connected.

\section{Extensions of the usual families}
\label{sec:extensions}
We give here two families of configurations, related to the families of Section \ref{sec:known_fam}. We first prove that a specific family, previously introduced by Carlitz and Scoville \cite{CarSco74}, is a specialisation of mixed Eulerian numbers $A_c(1)$ of connected configurations. As such, we give a natural $q$-analog of them.

In Section \ref{sec:almost_luka}, we define a new class of configurations, and give their remixed Eulerian number. The proof of this result, and the definition of this class, are reminiscent of {\L}ukasiewicz configurations. As such, we name this class: \emph{almost {\L}ukasiewicz configurations.} 

\subsection{Carlitz-Scoville generalised Eulerian numbers}
\label{sec:ab-eul}

We introduce two parameters for this section: $x$ and $y$, both being positive integers. 
We then define $A(r,s|x,y) := \frac{1}{(x+y-1)!} A_{(0^r,1^{y-1},r+s+1,1^{x-1},0^s)}(1).$ This notation is chosen to match Carlitz and Scoville's \emph{generalised Eulerian numbers} $A(r,s|\alpha,\beta)$ introduced in \cite{CarSco74}, later called \cite{Ji23} \emph{generalised Eulerian numbers}.

We first notice that $A(0,0|x,y) = 1,$ similarly to generalised Eulerian numbers. We then prove that our numbers satisfy the same induction relation \cite[Eq (1.9)]{CarSco74}, namely: 

\begin{lemma}
For all $x$, $y$, $r$, $s$ positive integers, we have:
\begin{equation}
\label{eq:a,b_recu}
A(r,s|x,y) = (s+x) A(r-1,s|x,y) + (r+y) A(r,s-1|x,y). 
\end{equation}
\end{lemma}
\begin{proof}
We first use the definition of our $A(r,s|x,y)$:
\begin{align*}
(x+y-1)! A(r,s|x,y) &= A_{(0^r,1^{y-1},r+s+1,1^{x-1},0^s)}(1) \\
&= A_{(0^r,1^{y-1},r+s,1^{x-1},0^s) + \{ r+y \} }(1) \\
&= (s+x) A_{(0^{r-1},1^{y-1},r+s,1^{x-1},0^s)}(1) + (r+y) A_{(0^r,1^{y-1},r+s,1^{x-1},0^{s-1})}(1) \\
&= (x+y-1)! ((s+x) A(r-1,s|x,y) + (r+y) A(r,s-1|x,y)).
\end{align*}
We used the final induction \eqref{eq:fin_ind} to get to line three.
\end{proof}

 These polynomials thus coincide with Carlitz-Scoville generalised Eulerian numbers:
 \begin{theorem}
\label{th:a,b_sub_acq}
The family $A(r,s|x,y)$ introduced by Carlitz and Scoville \cite{CarSco74} is a special case of remixed Eulerian numbers. More precisely, we have:
\begin{equation}
A(r,s|x,y) = \frac{1}{(x+y-1)!}A_{(0^r,1^{y-1},r+s+1,1^{x-1},0^s)}(1).
\end{equation}
Seeing as the relevant configurations are connected, we further note that the generalised Eulerian numbers are also a specialisation of the hit numbers of Garsia and Remmel.
\end{theorem}

We thus have a natural $q$-analog of these generalised Eulerian numbers:
\begin{equation}
A(r,s|x,y)_q := \frac{1}{[x+y-1]!} A_{(0^r,1^{y-1},r+s+1,1^{x-1},0^s)}(q).    
\end{equation}

\begin{proposition}
This q-analog satisfy a q-version of the induction equation \eqref{eq:a,b_recu}:
\begin{equation}
A(r,s|x,y)_q = q^{r+y}[s+x] A(r-1,s|x,y)_q + [r+y] A(r,s-1|x,y)_q.
\end{equation} \end{proposition}
\begin{proof}
We prove this by induction on $r+s$. The base case $r = s = 0$ follows from the definitions.

For the induction step, we follow the same path as the proof of the above Theorem \ref{th:a,b_sub_acq}: we write $A_{(0^r,1^{y-1},r+s+1,1^{x-1},0^s)}(q)$ as $A_{(0^r,1^{y-1},r+s,1^{x-1},0^s) + \{ r+y \}}(q)$, then use the final induction \eqref{eq:fin_ind}.
\end{proof}

\begin{proposition}
This $q$-analog has the following generating function, with fixed $r$ and $s$:
\begin{equation}
\frac{\sum_{i=0}^{r+s}{t^i A(i,r+s-i|x,y)_q}}{(t;q)_{r+s+x+y-1}} = \sum_{j \geq 0}{t^j {  \qbin{j+x+y-1}{j} [j+y}]^{r+s}}.
\end{equation}
Equivalently, extracting the coefficient in $t^r$ of the numerator, we get that for all $r,s$:
\begin{equation}
\label{eq:a,b_q}
A(r,s|x,y)_q = \sum_{j = 0}^{r}{(-1)^{r+j}q^{\binom{r-j}{2}} \qbin{j+x+y-1}{j} \qbin{r+s+x+y}{r-j}} [j+y]^{r+s}.
\end{equation}
\end{proposition}
\begin{proof}
We use Equalities \eqref{eq:acq-cnx-qpoch} and \eqref{eq:acq-cnx-1coeff} on $A_{(0^r,1^{y-1},r+s+1,1^{x-1},0^s)}(q).$
\end{proof}

We note that this Equality \eqref{eq:a,b_q} is a $q$-version of an equation of Carlitz and Scoville: \cite[Eq (3.14)]{CarSco74}.

\subsection{Almost {\L}ukasiewicz configurations}
\label{sec:almost_luka}
In this section, we study the family of \emph{almost {\L}ukasiewicz} configurations $c$. In order to define them, we remind the reader that a \emph{height} of a configuration $c$, at index $k$, was defined in Section \ref{sec:fin_ind} as $H_{c,k} = \sum_{i = 1}^{k}{(c_i-1)}.$
We then say that a configuration $c$ is almost {\L}ukasiewicz if all its heights $H_{c,k}$ are non-negative, except for exactly one, say $H_{c,j}$. The reader can convince itself that this height $H_{c,j}$ must then be equal to $-1$. We call $j$ the \emph{defect} of the configuration $c$.
\smallskip

One such configuration is $c = (1,0,3,0,1)$, pictured below, of defect $j = 2$.

\begin{figure}[!ht]
\centering
    \includegraphics[width=0.35\textwidth]{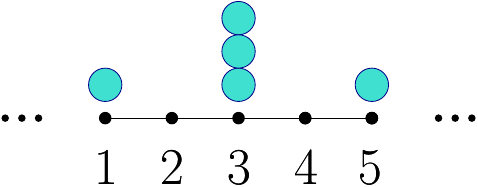}
\caption{An almost {\L}ukasiewicz configuration.\label{fg:alm_luka}}
\end{figure}

\begin{theorem}
\label{th:alm_luka}
Let $c$ be an almost {\L}ukasiewicz configuration, with $n$ balls and defect $j$. Then: $$A_c = \prod_{a \in \ms(c)}{[a]} - \qbin{n+1}{j} \prod_{a \in \ms(c), a < j}{[a]} \prod_{b \in \ms(c), b > j}{[b-j]}.$$
\end{theorem}
\begin{proof}
Let us consider $d$, the configuration obtained from $c$ by adding one ball onto site $1$. We write $\alpha = (c_1,...,c_{j-1})$ and $\beta = (c_{j+1},...,c_n)$, the sub-configurations of $c$ made of its balls left (resp. right) of the site $j$.

Using the final induction \eqref{eq:fin_ind} on $d$, then recalling that $\alpha$ and $\beta$ are {\L}ukasiewicz, we get:
\begin{align*}
A_d &= A_{c + \{ 1 \}} \\
&= [1] \qbin{n+1}{n+1} A_c + [1] \qbin{n+1}{j} A_{\alpha} A_{\beta} \\
&= A_c + \qbin{n+1}{j} \prod_{a \in \ms(\alpha)}{[a]} \prod_{b \in \ms(\beta)}{[b]} \\
&= A_c + \qbin{n+1}{j} \prod_{a \in \ms(c), a < j}{[a]} \prod_{b \in \ms(c), b > j}{[b-j]}.
\end{align*}

On the other hand, the configuration $d$ is {\L}ukasiewicz as well, because its heights $H_{d,j}$ all are one greater than the corresponding $H_{c,j}$, which are at least $-1$. 

We thus have $\displaystyle A_d = \prod_{a \in \ms(d)}{[a]} = [1] \prod_{a \in \ms(c)}{[a]}, \text{ QED.}$
\end{proof}

Applying this theorem to our example yields: $$A_{(1,0,3,0,1)} = [3]^3 [5] - \frac{[6][5][3]}{[2]} = 2q+6q^2+12q^3+16q^4+18q^5+16q^6+12q^7+6q^8+2q^9.$$

\section{Weakly {\L}ukasiewicz configurations}
\label{sec:wLuka}
In this section, we take $q > 0$: any ball has both a non-zero chance to bounce left, and to bounce right. We then define a class of configurations $c$ that behave nicely enough for their $A_c$ to be given a closed formula. More precisely, if the dynamic starts from such a configuration $c$ and end in the usual interval $[1;n]$, then we can guarantee that the set of fallen balls has been connected at all times. This will allow us to use the final induction \eqref{eq:fin_ind} with only two terms, leading to Theorem \ref{th:w_luka}.
\begin{definition}
A configuration is \emph{weakly {\L}ukasiewicz} iff, when dropping its balls left-to-right, either the support is guaranteed to be an interval at all times, or some ball falls outside of $[1;n].$
\end{definition}
\begin{proposition}[Characterisation of weakly {\L}ukasiewicz configurations]
\label{pr:wluka_carac}
A configuration $c$, of left-to-right order $u_1,u_2...,u_n$, is weakly {\L}ukasiewicz if and only if, for all $j \in [2;n],$ we have $u_j \leq max(u_{j-1}+1,j).$
\end{proposition}
\begin{proof}
We want to show that, supposing all the balls fall into $[1;n]$, the support is an interval at all times if and only if we have $u_j \leq max(u_{j-1}+1,j)$ for all $j \geq 2$. 
\smallskip

For the direct implication: suppose that, after the fall of $u_j$, the support has two intervals, whereas it only had one before that fall. This means that the $j-1$ previous ball did not fall on the site $u_j$, nor next to it. Since these $j-1$ previous balls started on the left of $u_j$, and none landed on $u_j$, they must have finished on the left of $u_j$, within a set of sites $I \subseteq [1;u_j-2].$

In order to have the space for all of these balls in this interval $I$, we need $j-1 \leq |I| \leq u_j-2$. Furthermore, for the previous ball $u_{j-1}$ to fall onto $I$, it needed to start within $I$. Therefore, $u_{j-1} \leq u_j-2$, and $u_j > max(u_{j-1}+1,j).$
\smallskip

Conversely, suppose that there is a ball $u_j > max(u_{j-1}+1,j).$ Then the first $j-1$ balls might fall onto $[1;j-1]$. In that case, $j-1$ and $u_j$ belong to the support, but not $j$, even though $j-1 < j < u_j.$ Therefore, the support might not be an interval after the fall of $u_j$. QED.
\end{proof}

Intuitively, in a weakly {\L}ukasiewicz configurations dropped left to right, each ball can only end next to the previous balls. In a connected configuration, the balls start next to each other, so they naturally end next to each other. In a {\L}ukasiewicz configuration, each ball end one site right of the previous one. 

\begin{lemma}
Let $c$ be a configuration, and $u_1,u_2...,u_n$ be its left-to-right order. Then:
\begin{itemize}
    \item $c$ is connected iff, for all $j \geq 2$, we have $u_j \leq u_{j-1}+1$.
    \item $c$ is \L ukasiewicz iff, for all $j \geq 2$, we have $u_j \leq j$.
\end{itemize}
\end{lemma}
\begin{proof}
    This result directly follows from the definitions of connected configurations and \L ukasiewicz configurations.
\end{proof}

Thus, the class of weakly {\L}ukasiewicz configurations contains both the connected and the {\L}ukasiewicz ones. However, some weakly {\L}ukasiewicz configurations are neither connected nor {\L}ukasiewicz, such as $c = (0,3,0,2,0).$

\begin{figure}[!ht]
\centering
    \includegraphics[width=0.35\textwidth]{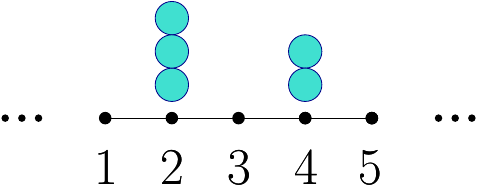}
\caption{A weakly {\L}ukasiewicz configuration.\label{fg:wluka}}
\end{figure}
\medskip

This class of configurations is stable under the operations of removing the rightmost ball and shifting all the balls to the left:
\begin{lemma}
\label{lm:wluka_stable}
Let $c = (0^k,\gamma,0^{n-m-k})$ be a weakly {\L}ukasiewicz configuration. Then the configurations $d = c-\{ k+m \}$ and $c^{(i)} := (0^i,\gamma,0^{n-m-i})$, where $i \leq k$, are also weakly {\L}ukasiewicz.
\end{lemma}
\begin{proof}
Thanks to Proposition \ref{pr:wluka_carac}, characterising weakly {\L}ukasiewicz configurations, we know that the left-to-right order $u_1,u_2...,u_n$ of $c$ satisfies $u_j \leq max(u_{j-1}+1,j)$ for all $j \in [2;n]$. We notice that $u_n$ is equal to $k + m$, i.e. the ball that we removed. The configuration $d$ thus has left-to-right order $u_1,u_2...,u_{n-1}$ for these same $u_j$. Therefore, the relations  $u_j \leq max(u_{j-1}+1,j)$ are still satisfied, for all $j \in [2;n-1]$.

As for a configuration $c^{(i)}$ with $i < k$, of left-to-right order $v_1, v_2..., v_n,$ we have the equality $v_j = u_j - (k-j)$ for all $j$. Therefore, $v_j \leq max(v_{j-1}+1,j).$ By the converse direction of Proposition \ref{pr:wluka_carac}, the configuration $c^{(i)}$ is weakly {\L}ukasiewicz.
\end{proof}

\begin{definition}
Let $k$ be an integer. We say that two series $\sum_{i \geq 0}{t^i a_i}$ and $\sum_{i \geq 0}{t^i b_i}$ are \emph{equal mod $t^k$} if, for all $i < k$, we have $a_i = b_i.$
\end{definition}

The previous lemma allows us to extend Theorem \ref{th:acq_cnx} for connected configurations onto this new class:
\thwluka

\begin{proof}
We remark that, in the proof of Theorem \ref{th:acq_cnx}, the only properties we used from the class of connected configurations were:
\begin{itemize}
    \item The right-most ball has to fall on either the left-most or right-most site;
    \item A connected configuration stays connected without its right-most ball.
    \item A connected configuration $ {c=(0^k,\gamma,0^{n-m-k})}$ stays connected when we move its core part $\gamma$ to the left.
\end{itemize}
Lemma \ref{lm:wluka_stable} ensures these properties are still valid on the class of weakly {\L}ukasiewicz configurations. Therefore, the proof of Theorem \ref{th:acq_cnx} still holds mod $t^{k+1}$.
\end{proof}
We can thus compute $A_{(0,3,0,2,0)} = [2]^3 [4]^2 - [3]^2 [6],$ similarly to $A_{(0,1,2,2,0)}.$

\section{One-hole configurations}
\label{sec:two_connec}

Seeing as a connected configuration is a configuration whose core part has no hole, it is natural to ask about the remixed Eulerian number of a configuration whose core part has exactly one hole. To this end, we define here the class of \emph{one-hole} configurations.

In order to provide a formula for their remixed Eulerian number in Theorem \ref{th:one_hole}, we slightly alter the generating function given by Theorem \ref{th:acq_cnx} for connected configurations.

\subsection{Enumeration of one-hole configurations}
\begin{definition}
Let $c$ be a configuration, with core part $\gamma$ as defined in Remark \ref{df:core_part}. We then define the \emph{corrective series} $P_\gamma$ with the following:
\begin{equation*}
    P_\gamma(q,t) = (t;q)_{n+1} \left( \sum_{j \geq 0}{t^j \prod_{a \in \ms(\gamma)}{[j+a]}} \right) - \sum_{i=0}^{n-m}{t^i A_{(0^i,\gamma,0^{n-\ell-m-1-i})}}.
\end{equation*}

Moving around the terms to match the shape of Equation \eqref{eq:acq-cnx-qpoch}, we get:
\begin{equation*}
    \frac{P_\gamma(q,t) + \sum_{i=0}^{n-m}{t^i A_{(0^i,\gamma,0^{n-\ell-m-1-i})}}}{(t;q)_{n+1}} = \sum_{j \geq 0}{t^j \prod_{a \in \ms(\gamma)}{[j+a]}},
\end{equation*}
\end{definition}

In this section, we will compute $P_\gamma(q,t)$ for some specific class of $\gamma$, proving along the way that they are polynomials.

We can rewrite Theorem \ref{th:acq_cnx} for connected configurations as: if $\gamma$ has at least one ball per site, i.e. $c$ is connected, then $P_\gamma(q,t) = 0$. Similarly, Theorem \ref{th:w_luka} over weakly {\L}ukasiewicz configurations becomes: $P_\gamma(q,t)$ is divisible by $t^{k+1}$ if $c$ starts with $k$ zeroes while being weakly {\L}ukasiewicz.

\begin{definition}
A configuration $c$ with core $\gamma$ is \emph{one-hole} if $\gamma$ has exactly one empty site. 
 \end{definition}

In this case, $\gamma$ has the form $(\alpha,0,\beta)$, with $n$ total balls. The sub-configuration $\alpha =  (\alpha_1,...,\alpha_{\ell})$ has $p$ balls on $\ell > 0$ sites, each site having at least one ball. Similarly, the sub-configuration $\beta = (\beta_1,...,\beta_m)$ comprises $r > 0$ balls on $m > 0$ sites, each site having at least one ball. We will use these notations for the remainder of this section.

One such configuration is $c = (0,2,1,0,3,0)$, with $\alpha = (2,1), \beta = (3), \ell = 2, m = 1, p~=~3,$
 $r~=~3$, cf Figure \ref{fg:one-hole}.
 
\begin{figure}[!ht]
\centering
    \includegraphics[width=0.4\textwidth]{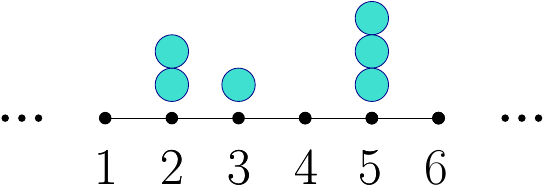}
\caption{A one-hole configuration with core $(2,1,0,3)$.\label{fg:one-hole}}
\end{figure}

For these one-hole configurations $c$ with core part $\gamma$, we give a closed formula of the corrective series $P_\gamma(q,t)$:
\begin{theorem}
\label{th:one_hole}
Let $\gamma$ be a core part with exactly one empty site. Then:
\begin{equation}
    \label{eq:acq-1hole}
    P_\gamma(q,t) = t^{p-\ell} \left( \prod_{a \in \ms(\alpha)}{[\ell+1-a]q^{a-\ell}} \right) \left( \prod_{b \in \ms(\beta)}{[b]} \right)
    \left( \sum_{i = 0}^{r}{(-1)^i t^{i}q^{\binom{p}{2}+pi+\binom{i+1}{2}} \qbin{n+1}{r-i}} \right).
\end{equation}
Extracting a single coefficient of the numerator, we get that for all $i \in [0;n-\ell-m-1]$:
\begin{align}
    A_{(0^i,\gamma,0^{n-\ell-m-1-i})} &= \sum_{j = 0}^{i}{(-1)^{i+j}q^{\binom{i-j}{2}}\qbin{n+1}{i-j}} \prod_{a \in \ms(\gamma)}{[j+a]} \\
    &+ [\! [i \geq p-\ell ]\! ] (-1)^{i+p+\ell+1} q^{\binom{i+\ell}{2}} \! \prod_{a \in \ms(\alpha)}{\! [\ell+1-a]q^{a-\ell}} \hspace{-5pt} \prod_{b \in \ms(\beta)}{\hspace{-5pt} [b]} \qbin{n+1}{i+\ell+1}. \nonumber
\end{align}
where $[\! [i \geq p-\ell ]\! ]$ is $1$ if $i \geq p-\ell$, and $0$ otherwise.
\end{theorem}
We dedicate the next subsection to the proof of this theorem.
\smallskip

On our example $c = (0,2,1,0,3,0)$, we thus have:
\begin{align*}
&A_{(0,2,1,0,3,0)} = -[7][1]^2[2][4]^3+[2]^2[3][5]^3-q^{3-2}[2]^2\qbin{7}{4} = [2]^2[3][5]^3 - [2][4]^3[7] - q\frac{[2][5][6][7]}{[3]} \\
&= 2q^2+8q^3+19q^4+36q^5+56q^6+72q^7+78q^8+72q^9+56q^{10}+36q^{11}+19q^{12}+8q^{13}+2q^{14}.
\end{align*}
\medskip

\begin{proposition}
Let $c$ be an one-hole configuration. Then either $c$ or its reverse $\overline{c}$ is weakly {\L}ukasiewicz.
\end{proposition}
\begin{proof}[Proof]
    Say the hole is in site $k$, and that they are $p$ balls on its left. We can assume, up to reversing left and right, that $p \geq k$. Then, for the left-to-right order $u_1,u_2...,u_n$, we have $u_j \leq u_{j-1}+1$ for all $j \geq 2$ except $p+1$. For this value $j = p+1$, we have $u_j = k+1 \leq p+1 = j$, as this ball is the leftmost one among those starting right of the hole. 
    
    This configuration thus satisfies the characterisation of weakly {\L}ukasiewicz configurations given by Proposition \ref{pr:wluka_carac}.
\end{proof}

Therefore, after maybe using Proposition \ref{pr:reverse} over reverse configurations, we have a formula for $A_c$ using Theorem \ref{th:w_luka} over weakly {\L}ukasiewicz configurations. However, the formula in Theorem \ref{th:one_hole} is still of interest, as its sum ranges over far fewer indices in some cases, and its proof uses new techniques.

As an example, on the configuration $(0,0,3,0,r,0^{r-2})$, this theorem yields a formula with only three terms, instead of the $r-2$ terms given by Theorem \ref{th:w_luka} over the weakly {\L}ukasiewicz reverse configuration $(0^{r-2},r,0,3,0,0)$.

\subsection{Proof of Theorem \ref{th:one_hole}}
We write $P_\gamma$ instead of $P_\gamma(q,t),$ seeing that the values for $q$ and $t$ are fixed.

We break up this proof by first stating two lemmas. The first one links the corrective series $P_\gamma$ of all the one-hole configurations that have the same amount of balls left and right of the hole, written respectively $p$ and $r$. The second one gives an explicit formula for the corrective series of a specifically chosen configuration, with given $p$ and $r$, using Theorem \ref{th:alm_luka} over almost {\L}ukasiewicz configurations.
\medskip

We first assert that, at fixed $p$ and $r$, we only have to prove Theorem \ref{th:one_hole} for a single pair $(\alpha,\beta)$, thanks to the following lemma:
\begin{lemma}
\label{lm:1hole_2cols_enough}
 We have the following equality:
\begin{align}
P_{(\alpha,0,\beta)} &= P_{(p,0,r)} \prod_{a \in \ms(\overline{\alpha})}{[a]q^{-a+1}} \prod_{b \in \ms(\beta)}{[b]} \nonumber \\
&= P_{(p,0,r)} \prod_{a \in \ms({\alpha})}{[\ell+1-a]q^{a-\ell}} \prod_{b \in \ms(\beta)}{[b]}.
\end{align}
\end{lemma}

\begin{proof}
    We prove this lemma by induction on the amount of balls in $(\alpha_1,...,\alpha_{\ell-1})$, i.e. in $\alpha$ except its right-most site. The induction on the number of balls in $(\beta_2,...,\beta_m)$ is similar. The base case of zero balls is direct.
    
    We write $\delta = (\alpha_2,...,\alpha_{\ell-1})$ and $\cst = n-\ell-m-1$. We then define the configurations $L = (\alpha_1,\delta,\alpha_\ell,0,\beta)$, $M = (\alpha_1-1,\delta,\alpha_\ell,1,\beta)$ and $R = (\alpha_1~-~1, \delta, \alpha_\ell+1, 0, \beta)$. Our induction step will aim to compute $P_L$, given $P_R$. We distinguish two cases.
    
    If $\alpha_1 > 1$, the induction step we have to prove simplifies down to $q^{\ell-1}P_L = [\ell] P_R$. We then remark that, for any $j$, the following holds: $q^{\ell-1}[j+1]-[\ell][j+\ell] = - [\ell-1][j+\ell+1]$. Unfolding the definition of $P_L$ and $P_R,$ we have:

\begin{align*}
    q^{\ell-1}P_{L} - [\ell] P_{R} &= (t;q)_{n+1}\sum_{j \geq 0}{t^j (q^{\ell-1}[j+1]-[\ell][j+\ell]) \prod_{d \in \ms(\alpha_1-1,\delta,\alpha_\ell,0,\beta)}{[j+d]}} \\
    &+ \sum_{i=0}^{\cst}{t^i (-q^{\ell-1}A_{(0^i,L,0^{\cst-i})} + [\ell] A_{(0^i,R,0^{\cst-i})})} \\
    &= -[\ell-1](t;q)_{n+1}\sum_{j \geq 0}{t^j [j+\ell+1] \prod_{d \in \ms(\alpha_1-1,\delta,\alpha_\ell,0,\beta)}{[j+d]}} \\
    &+ [\ell] \sum_{i=0}^{\cst}{t^i \left( A_{(0^i,R,0^{\cst-i})} -q^{\ell-1} \frac{1}{[\ell]} A_{(0^i,L,0^{\cst-i})} \right) } \\
    &= [\ell] \sum_{i=0}^{\cst}{t^i \left( A_{(0^i,R,0^{\cst-i})} - q^{\ell-1}\frac{1}{[\ell]} A_{(0^i,L,0^{\cst-i})} - \frac{[\ell-1]}{[\ell]}A_{(0^i,M,0^{\cst-i})} \right) } \\
    &= 0.
\end{align*}

We obtain the second equality from Theorem \ref{th:acq_cnx} over the connected configuration M, and the final one from Equality \eqref{eq:big_step} on each $(0^i,R,0^{\cst-i})$.
\medskip

    In the case $\alpha_1 = 1$, we need to prove $tq^{\ell-1} P_L = [\ell] P_R$. We first repeat the above computations, leaving us only to show that
$P_{(0,\delta,f+1,0,\beta)}$ is indeed equal to $t^{-1} P_{(\delta,f+1,0,\beta)}$. We thus unfold the definition of $P_{(\delta,f+1,0,\beta)}$:
\begin{align*}
    t^{-1}P_{(\delta,f+1,0,\beta)} &= (t;q)_{n+1} \sum_{j \geq 0}{t^{j-1}} \prod_{d \in \ms(\delta,f+1,0,\beta)}{[j+d]} - \sum_{i = 0}^{\cst}{t^{i-1}A_{(0^i,\delta,f+1,0,\beta,0^{\cst-i})}} \\
    &= (t;q)_{n+1} \sum_{j \geq 1}{t^{j-1}} \prod_{d \in \ms(\delta,f+1,0,\beta)}{[j+d]} - \sum_{i = 1}^{\cst+1}{t^{i-1}A_{(0^{i},0,\delta,f+1,0,\beta,0^{\cst-i})}} \\
    &+ t^{-1}\prod_{d \in \ms(\delta,f+1,0,\beta)}{[d]} - t^{-1}A_{(\delta,f+1,0,\beta,0^{\cst})} \\
    &= (t;q)_{n+1} \sum_{j \geq 1}{t^{j-1}} \prod_{d \in \ms(0,\delta,f+1,0,\beta)}{[j+d-1]} - \sum_{i = 1}^{\cst+1}{t^{i-1}A_{(0^{i},0,\delta,f+1,0,\beta,0^{\cst-i})}} \\
    &= P_{(0,\delta,f+1,0,\beta)}.
\end{align*}
We use Proposition \ref{pr:acq_luka} on the {\L}ukasiewicz configuration $(\delta,f+1,0,\beta,0^{cst})$ to prove the third equality; the final one is obtained through an index shift.
\end{proof}

Thus for fixed $p$ and $r$, we only have to prove Theorem \ref{th:one_hole} for a single pair $(\alpha,\beta).$

\begin{lemma}
\label{lm:correctif_power_minus_bin}
For all $p \geq 2$ and $r \geq 1$, we have:
\begin{equation*}
A_{(0^{p-2},p,0,1^{r-1})} = q^{\frac{p(p-3)}{2}}[r-1]! \left( [r+1]^p-\qbin{p+r}{r} \right) .
\end{equation*}
\end{lemma}
\begin{proof}
    We first notice that the reverse configuration $(1^{r-1},0,p,0^{p-2})$ is almost {\L}ukasiewicz, of defect $r$. Therefore, thanks to Theorem \ref{th:alm_luka} over almost {\L}ukasiewciz configurations, we have: $$A_{(1^{r-1},0,p,0^{p-2})} = [r+1]^p [r-1]! - \qbin{p+r}{r} [r-1]!.$$
    
    We then recall Proposition \ref{pr:reverse} over reverse configurations, and the identity $[i]_q = q^{-i+1}[i]_{1/q},$ giving us:
    \begin{align*}
    A_{(0^{p-2},p,0,1^{r-1})} &= q^{\binom{p+r-1}{2}} A_{(1^{r-1},0,p,0^{p-2})}(1/q) \\
    &= q^{\binom{p+r-1}{2}} q^{-\binom{r-1}{2}} [r-1]! \left( q^{-rp} [r+1]^p - q^{-\binom{p+r}{2}+\binom{r}{2}+\binom{p}{2}}\qbin{p+r}{r} \right)\\
    &= q^{\frac{p(p-3)}{2}}[r-1]! \left( [r+1]^p-\qbin{p+r}{r} \right).
    \end{align*}
\end{proof}

We can now prove Theorem \ref{th:one_hole}. We start by using Lemma \ref{lm:1hole_2cols_enough}, to assume w.l.o.g. that $\gamma = (p,0,1^r)$. We then compute $P_{(p,0,1^r)}$ by induction over $r$, starting by the use of the equality $[j+r+2] = [r] + q^r [j+2]$:
\begin{align*}
    P_{(p,0,1^r)} &= (t;q)_{p+r+1} \sum_{j \geq 0}{t^j [j+1]^p \prod_{a = 3}^{r+2}{[j+a]}} - \sum_{i = 0}^{p-2}{t^i A_{(0^i,p,0,1^r,0^{p-2-i})}} \\
    &= [r] (t;q)_{p+r+1} \sum_{j \geq 0}{t^j [j+1]^p} \prod_{a = 3}^{r+1}{[j+a]} + q^r (t;q)_{p+r+1} \sum_{j \geq 0}{[j+1]^p} \prod_{a = 2}^{r+1}{[j+a]} \\
    &- \sum_{i = 0}^{p-2}{t^i A_{(0^i,p,0,1^r,0^{p-2-i})}}. 
\end{align*}
The sub-configuration $(p,1^r)$ has no hole, so we can use Theorem \ref{th:acq_cnx} for connected configurations in the second term of this sum. As for the first term, we rewrite it using the definition of $P_{(p,0,1^r)}:$
\begin{align*}
    P_{(p,0,1^r)} &= - \sum_{i = 0}^{p-2}{t^i A_{(0^i,p,0,1^r,0^{p-2-i})}} + q^r \sum_{i = 0}^{p-1}{t^i A_{(0^i,p,1^r,0^{p-1-i})}} + [r](1-tq^{p+r}) P_{(p,0,1^{r-1})} \\
    &+ [r] \sum_{i = 0}^{p-2}{t^i A_{(0^i,p,0,1^{r-1},0^{p-2-i})}} - [r]q^{p+r} \sum_{i = 1}^{p-1}{t^i A_{(0^{i-1},p,0,1^{r-1},0^{p-1-i})}} \\
    &= -A_{(p,0,1^r,0^{p-2})} + q^r A_{(p,1^r,0^{p-1})} + [r]A_{(p,0,1^{r-1},0^{p-2})} \\
    &+ \sum_{i = 1}^{p-2}{t^i(q^r A_{(0^i,p,1^r,0^{p-1-i})} + [r] A_{(0^i,p,0,1^{r-1},0^{p-2-i})} - [r]q^{p+r} A_{(0^{i-1},p,0,1^{r-1},0^{p-1-i})} - A_{(0^i,p,0,1^r,0^{p-2-i})} )} \\
    &+ t^{p-1}(q^r A_{(0^{p-1},p,1^r)} - [r] q^{p+r} A_{(0^{p-2},p,0,1^{r-1})}) + [r](1-tq^{p+r})P_{(p,0,1^{r-1})}.
    \end{align*}

In that last expression, it turns out that the first two lines are zero. Indeed, using Theorem \ref{th:w_luka} over weakly {\L}ukasiewicz configurations on $(0^{p-2},p,0,1^r)$, we know that $P_{(p,0,1^r)}$ is divisible by $t^{p-1}$. We thus have, using Lemma \ref{lm:correctif_power_minus_bin}:
    \begin{align*}
    P_{(p,0,1^r)} &= t^{p-1}q^r \left( q^{\binom{p}{2}} [r+1]^p [r]! - [r]q^p q^{\frac{p(p-3)}{2}}[r-1]! \left( [r+1]^p-\qbin{p+r}{r} \right) \right) \\
    &+ [r](1-tq^{p+r})P_{(p,0,1^{r-1})} \\
    &= t^{p-1}q^{r+\binom{p}{2}}[r]!\qbin{p+r}{r} + [r](1-tq^{p+r})P_{(p,0,1^{r-1})}.
    \end{align*}

Another application of Lemma \ref{lm:1hole_2cols_enough} gives us $[r]! P_{(p,0,r)} = P_{(p,0,1^r)}$. We normalise $P_{(p,0,r)}$ as the polynomial $$Q_r := t^{-p+1}q^{-\binom{p}{2}} P_{(p,0,r)}.$$
The above equality then simplifies down to: $$Q_r = q^{r}\qbin{p+r}{r} + (1-tq^{p+r})Q_{r-1}.$$ Furthermore, $P_{(p,0,0)} = P_{(p)}$, which we know to be $0$ thanks to Theorem \ref{th:acq_cnx} for connected configurations, applied on $\beta = (p)$. Therefore, $Q_0 = 0$.
\medskip

In order for Theorem \ref{th:one_hole} to be true, we need $Q_r$ to be equal to $$\Tilde{Q}_r := \sum_{i=0}^r{(-t)^i q^{pi + \binom{i+1}{2}} \qbin{p+r+1}{r-i}}.$$ We see that $\Tilde{Q}_0$ is indeed $Q_0$, i.e. $0$. Moreover, we check that $\Tilde{Q}_r$ satisfies the same induction relation:
    \begin{align*}
    q^r \qbin{p+r}{r}+(1-tq^{p+r})\Tilde{Q}_{r-1} &= q^r \qbin{p+r}{r} + \sum_{i=0}^{r-1}{(-t)^i q^{pi+\binom{i+1}{2}} \qbin{p+r}{r-i-1}} \\
    &+ \sum_{i=0}^{r-1}{(-t)^{i+1} q^{p(i+1)+\binom{i+1}{2}+r} \qbin{p+r}{r-(i+1)}} \\
    &= q^r \qbin{p+r}{r} + \qbin{p+r}{r-1} + (-t)^r q^{pr+\binom{r+1}{2}} \\
    &+ \sum_{i=1}^{r-1}{(-t)^i q^{pi+\binom{i+1}{2}} \left( \qbin{p+r}{r-i-1} + q^{r-i} \qbin{p+r}{r-i} \right) } \\
    &= \Tilde{Q}_r, \text{ QED}.
    \end{align*}
We thus have $Q_r = \Tilde{Q}_r = \sum_{i=0}^r{(-t)^i q^{pi + \frac{(i+1)i}{2}} \qbin{p+r+1}{r-i}}$, which concludes the proof.


\printbibliography

\end{document}